\documentclass[11pt]{amsart}

\usepackage{epsfig}

\usepackage{graphicx}
\usepackage{amscd}
\usepackage{amsmath}
\usepackage{amsfonts}
\usepackage{psfrag}
\usepackage{amssymb}
\usepackage{verbatim}
\usepackage{comment}

\newtheorem{theorem}{Theorem}[section]
\theoremstyle{plain}

\newtheorem{corollary}[theorem]{Corollary}

\newtheorem{example}[theorem]{Example}

\newtheorem{lemma}[theorem]{Lemma}

\def\Pref{{\rm Pref}}

\newcommand{\gam}{\gamma}
\newcommand{\om}{\omega}
\def\Pmu{\P_{\!\!\mu}}
\def\Om{\Omega}
\newcommand{\Sig}{\Sigma}

\newcommand{\Gam}{\Gamma}

\newcommand{\R}{{\mathbb R}}

\newcommand{\Z}{{\mathbb Z}}

\def\N{{\mathbb N}}

\newcommand{\Prob}{{\mathbb P}\,}
\def\P{\Prob}

\def\Av{{\rm Av}}
\def\Ak{{\mathcal A}}

\def\be{\begin{equation}}
\def\ee{\end{equation}}
\newcommand{\Ek}{{\mathcal E}}

\newcommand{\eps}{{\varepsilon}}
\newcommand{\es}{\emptyset}

\def\ov{\overline}
\def\und{\underline}

\begin{document}

\title[Fractals defined via the semigroup generated by 2 and 3]{Dimensions of some fractals defined via the semigroup generated by 2 and 3}

\author{Yuval Peres}
\address{Yuval Peres\\ Microsoft Research, One Microsoft Way, Redmond, WA 981052, USA}
\email{peres@microsoft.com}
\author{Joerg Schmeling}
\address{Joerg Schmeling, Lund University, Sweden}
\email{joerg@maths.lth.se}
\author{St\'ephane Seuret}
\address{St\'ephane Seuret, LAMA UMR CNRS 8050, Universit\'e Paris-Est Cr\'eteil, 940010 Cr\'eteil Cedex - France}
\email{seuret@u-pec.fr}
\author{Boris Solomyak }
\address{Boris Solomyak, Box 354350, Department of Mathematics,
University of Washington, Seattle, WA 98195, USA}
\email{solomyak@math.washington.edu}

\date{\today}

\begin{abstract}
We compute the Hausdorff and Minkowski dimension of subsets of the symbolic space $\Sigma_m=\{0,...,m-1\}^{\mathbb{N}}$ that are invariant under multiplication by integers. 
The results apply to the sets $\{x\in \Sigma_m:  \forall\, k, \ x_kx_{2k}\cdots x_{n k}=0\}$, where  $n\ge 3$.
We prove that for such sets, the Hausdorff and Minkowski dimensions typically differ.

\end{abstract}

\maketitle

\section{Introduction}

Let $m\ge 2$ be an integer. A widely studied issue in dynamics consists in computing the (Hausdorff, Minkowski,\ldots) 
dimensions of subsets $X$ of  the symbolic space $\Sig_m = \{0,\ldots,m-1\}^{\N}$. 
When $X$ is a closed subset of $\Sigma_m$, invariant under the shift $x\mapsto mx$, by a well-known result of Furstenberg \cite{FURS}, 
both the Hausdorff and Minkowski dimensions of $X$ coincide with the topological entropy of the shift on $X$ divided by  $\log m$. 
This theorem covers a lot of interesting examples. Unfortunately, as soon as the set is not invariant any more, many standard techniques fail, the most 
basic example of which is
$$X_2=\{x=(x_k)_{k\geq 1}\in \Sigma_2: \, \forall k\geq 1, \ x_k x_{2k} =0\}.$$
In  \cite{KPS}, the dimensions of $X_2$ were computed. 
In particular, it is shown that the Hausdorff dimension of $X_2$ is strictly smaller than its  Minkowski dimension, this being a reflection
of the ``non-self-similarity'' resulting from its definition.

In \cite {KPS},  the key property used to study $X_2$ is that this set, though not invariant under the shift,
is nevertheless invariant under the action of multiplicative integers. More precisely, 
$$x=(x_k)_{k\geq 1} \in X_2 \ \Longrightarrow \ \forall \, i\in \mathbb{N}, \ (x_{in})_{n\geq 1} \in X_2.$$
In \cite{KPS} the Hausdorff dimension of $X_2$ and of many more general sets invariant under 
the action of multiplicative integers were computed.

The techniques developed in  \cite {KPS}, however, do not allow directly to determine the dimensions of sets such as
$$X_{2,3}=\{x=(x_n)_{n\geq 1}\in \Sigma_2: \, \forall k\geq 1, \ x_k x_{2k} x_{3k} =0\},$$
which is  also invariant under the multiplication by integers. 
Roughly speaking, \cite{KPS} relied on the fact that the condition $(x_j)_{j\ge 1}\in X_2$ ``splits'' into independent conditions along
geometric progressions of ratio 2, namely, that the sequence $(x_{i 2^k})_{k\ge 0}$ contains no two consecutive 1's for any odd $i$.
 In order to understand the structure of $X_{2,3}$, we will need to work with the semigroup generated by 
2 and 3 instead of the cyclic semigroup $\{2^k\}_{k\ge 0}$.

\medskip

Finding the dimensions of sets like $X_{2,3}$ is related to the general question of 
{\em multiple ergodic averages}: let $T:X\to X$ be a dynamical system, and $f:X^\ell\to \R$ a H\"older continuous potential 
($\ell\geq 1$ being an integer). Classical questions concern the possible limits, for $x\in X$,  
of the multiple ergodic averages defined  by
\be \label{mult-erg}
S_n  \, f (x) =  \frac{1}{n} \sum_{k=0}^{n-1} f(T^kx,T^{2k}x, ... , T^{\ell k}x),
\ee
when $n$ goes to infinity. Furstenberg, see \cite{FURS2}, 
introduced such non-conventional ergodic sums in his proof of the existence of arithmetic progressions
of arbitrary length in sets of positive density (Sz\'emeredi Theorem). 
A natural extension of classical multifractal analysis consists in  investigating the (dimensions of the) sets 
$$E_f(\alpha):= \{x: \lim_{n\to +\infty} S_n\, f(x) =\alpha\}.$$
These questions have been investigated  by many authors, see e.g.\ \cite{BOUR,HOSTKRA}, and more recently in \cite{kifer,FLM,KPS,PS,FSW}. 
Our set  $X_{2,3}$ is contained in, and can be shown to have the same dimension as, the set $E_f(0)$ in the simple case where $\ell=3$ and 
$f(x,y,z)=x_1y_1z_1$. In fact, the paper \cite{FLM}, which raised the question of computing the Hausdorff dimension of $X_2$, 
was the motivation for \cite{KPS}, where $\dim_H(X_2)$ was determined. 
In turn, the authors of \cite{FSW}, building in part on \cite{KPS},  were able to compute the multifractal spectrum
of certain ``double'' ergodic averages, that is, when $\ell=2$ in (\ref{mult-erg}). (Independently, some special cases were handled in \cite{PS}.)
We hope that the methods developed in the current paper
will make it possible to perform a similar analysis for an arbitrary $\ell\in \N$.

\medskip

The goal of this paper is to understand the structure of the sets, such as
\be \label{def-xgen}
 X^{(m)}_{n_1,\ldots,n_r}:=
\left\{{(x_k)}_{k=1}^\infty \in \Sig_m:\ x_{kn_1} x_{kn_2} \cdots x_{kn_r}=0, \ \forall\,k\in \N \right\},
\ee
where $n_1,\ldots,n_r$ are arbitrary distinct positive integers, in particular,
\be \label{def-xxn}
X_{2,3,\ldots,n}=X_{2,3,\ldots,n}^{(2)}=  \left\{{(x_k)}_{k=1}^\infty \in \Sig_2:\ x_k x_{2k} \cdots x_{nk}=0, \ \forall\,k\in \N \right\}.
\ee
First, we represent the Minkowski dimension as the sum of a series. For the Hausdorff dimension, we obtain, on the one hand, a variational formula; and on the other hand, a formula based
on a system of nonlinear equations on an infinite tree.  This tree has levels naturally indexed by a sub-semigroup of the multiplicative positive integers (e.g. the semigroup generated by 2 and 3).
The formulas are complicated (more so than in \cite{KPS}), but this seems
unavoidable. In any case, they allow reasonably accurate numerical estimates. Perhaps more importantly, they yield a
qualitative result: the Hausdorff dimension is strictly less than the Minkowski dimension for all sets of the form
(\ref{def-xgen}).

The paper is organized as follows. 
In Section 2 we present precise statements of the results (Theorems \ref{th-mink} and \ref{th-var}).  Section 3  contains some preliminary results, 
Sections 4, 5 and 6  give the proofs: the Minkowski dimension in Section 4, and the lower and upper bound for the Hausdorff dimension in Sections 5 and 6 respectively. Finally, Section 7 contains some numerical estimations and  further examples.

\section{Statement of results}

Let   $J\ge 2$ be an integer. 
Consider the semigroup $S=\langle p_1,\ldots,p_J\rangle$ 
generated by distinct  primes $p_1,\ldots,p_J$. Denote by $\ell_k$ the $k$-th element of $S$, so that 
$$
S = \{\ell_k\}_{k=1}^\infty,\ \ 1=\ell_1 < \ell_2 < \ldots
$$

\medskip

\noindent {\bf Notation.} We write
$
(i,S)=1
$
if and only if $p_j\nmid i$ for all $j\le J$ (in other words, $i$ is mutually prime with all elements of $S$).
Observe that
\be \label{eq-disj}
\N = \bigsqcup\{iS:\ (i,S)=1\}
\ee
is a disjoint union. 

\medskip

To each element $x =
{(x_k)}_{k=1}^\infty \in \Sig_m$, one can associate the subsequence  $x|_{iS}$, viewed as an element of $\Sig_m$, defined as  
$$
x|_{iS}:={(x_{i\ell_k})}_{k=1}^\infty.
$$ 

Given a closed subset $\Omega\subset\Sigma_m$ let
\begin{equation} \label{def-Xom}
X_\Om^{(S)}:= \Bigl\{x =
{(x_k)}_{k=1}^\infty \in \Sig_m:\ x|_{iS} \in \Om\ \ \mbox{for all}\ i,\ (i,S)=1\Bigr\}.
\end{equation}

In this article, we obtain formulas for the Minkowski and Hausdorff dimensions of $X_\Om^{(S)}$. Note that the case when $J=1$ (the 
semigroup $S$ is cyclic) was considered in \cite{KPS}. Our main example is the set from the Introduction
$$
X_{2,3}:= \left\{{(x_k)}_{k=1}^\infty \in \Sig_2:\ x_k x_{2k} x_{3k}=0, \ \forall\,k\in \N \right\}=X_\Om^{(S)},
$$
for which $S$ is the semigroup generated by 2 and 3, and
$$
\Om=\left\{{(\om_k)}_{k=1}^\infty  \in \Sig_2:\ \forall \, i\geq 1, \ \om_i \om_j \om_k=0\ \mbox{if}\ 2\ell_i=\ell_j,\ 3 \ell_i = \ell_k\right\}.
$$
More generally, the sets  $X_{2,3,\ldots,n}$ defined by (\ref{def-xxn}),
correspond to the case where  $S$ is the semigroup generated by all primes less than or equal to $n$ and
$$
\Om=\left\{{(\om_k)}_{k=1}^\infty  \in \Sig_2:\ \om_{i_1} \om_{i_2}\cdots \om_{i_n}=0\ 
\mbox{if}\ j\ell_{i_1}=\ell_{i_j},\ j=1,\ldots,n \right\}.
$$
Even more generally, our set-up includes the sets defined in (\ref{def-xgen}):
$$
X^{(m)}_{n_1,\ldots,n_r}=
\left\{{(x_k)}_{k=1}^\infty \in \Sig_m:\ x_{kn_1} x_{kn_2} \cdots x_{kn_r}=0, \ \forall\,k\in \N \right\}
$$
for arbitrary $n_1,\ldots,n_r\in \N$. In fact, $X^{(m)}_{n_1,\ldots,n_r}=X_\Om^{(S)}$,
where $S$ is the semigroup generated by all prime factors of the numbers 
$n_1,\ldots,n_r$ and 
$$
\Om=\left\{{(\om_k)}_{k=1}^\infty  \in \Sig_m:\ \om_{i_1} \om_{i_2}\cdots \om_{i_r}=0\ 
\mbox{if}\ n_j\ell_{i_1}=\ell_{i_j},\ j=1,\ldots,r\right\}.
$$

Throughout the paper, we fix the standard metric on $\Sig_m$: 
$$ 
\varrho((x_k)_{k\ge 1}, (y_k)_{k\ge 1})=m^{-\min\{n:\ x_n\ne y_n\}}.
$$ 
All the dimensions are computed with respect to this metric. 
It is well-known that if we map $\Sig_m$ onto $[0,1]$ via the base-$m$ expansion, the dimensions of a subset of $(\Sig_m,\varrho)$ and its
image on the real line (with respect to the Euclidean metric) are the same.

\medskip

Next we continue with the general set-up and
consider the tree of prefixes of the set $\Om$. It is a directed graph $\Gam=\Gam(\Om)$ whose
set of vertices is
$$
V(\Gam) = \Pref(\Om) = \bigcup_{k=0}^\infty \Pref_k(\Om),
$$
where $\Pref_0(\Om)$ has only one element, the empty word $\varnothing$, and
$$\Pref_k(\Om) = \left\{u\in \{0,\ldots,m-1\}^k,\ \Om \cap [u]\ne \es\right\}.$$
Here and below we denote by $[u]$ the cylinder set of all sequences starting with $u$.
There is a directed edge from a prefix $u$ to a prefix $v$ if $v=ui$ for some $i\in \{0,\ldots,m-1\}$.
In addition, there is an edge from $\varnothing$ to every $i\in \Pref_1(\Om)$.
Clearly, $\Gam(\Om)$ is a tree, and there is at least one edge going out of every vertex.
Denote
\be
\label{defak}
A_k = |\Pref_k(\Om)|.
\ee
Let 
\be \label{eq-fs}
\gam(S):= \sum_{k=1}^\infty \frac{1}{\ell_k}\,.
\ee
Observe that 
\be \label{eq-fs2}
\gam(S) = \prod_{j=1}^J \sum_{k=0}^\infty \frac{1}{p_j^k}= \prod_{j=1}^J {\Bigl(1-\frac{1}{p_j}\Bigr)}^{-1}.
\ee

\begin{theorem} \label{th-mink}
The Minkowski dimension of the set $X_\Om^{(S)}$, defined by (\ref{def-Xom}), equals
\begin{eqnarray*}
\dim_M(X_\Om^{(S)})  & = & \gam(S)^{-1} \sum_{k=1}^\infty \log_m(A_k)\Bigl(\frac{1}{\ell_k} - \frac{1}{\ell_{k+1}}\Bigr) \\
                     & = & \gam(S)^{-1} \Bigl( 1 + \sum_{k=1}^\infty \frac{\log_m(A_{k+1}/A_k)}{\ell_{k+1}}\Bigr).
\end{eqnarray*}
\end{theorem}

The first formula for the Hausdorff dimension is obtained via
a version of the variational principle. Let $\mu$ be a Borel probability 
measure on $\Om$. Denote by $\alpha_k$ the partition of $\Om$ into cylinders of length $k$, so that $A_k = |\alpha_k|$.
We consider the Shannon entropy of a finite partition, using logarithms based $m$:
$$
H^\mu(\alpha):= -\sum_{B\in \alpha} \mu(B)\log_m\mu(B)
$$
and the conditional entropy $H^\mu(\alpha|\beta)$ for two finite partitions. Define
\begin{eqnarray} 
s(\Om,\mu) & := & \gam(S)^{-1} \sum_{k=1}^\infty H^\mu(\alpha_k)\Bigl(\frac{1}{\ell_k} - \frac{1}{\ell_{k+1}}\Bigr) \nonumber
         \\  &  = & \gam(S)^{-1} \Bigl( H^\mu(\alpha_1) + \sum_{k=1}^\infty \frac{H^\mu(\alpha_{k+1}|\alpha_k)}{\ell_{k+1}}\Bigr).
\label{def-som}
\end{eqnarray} 

\begin{theorem}\label{th-var}
{\bf (i)} We have
$$
\dim_H(X_\Om^{(S)}) = \sup_\mu s(\Om,\mu),
$$
where the supremum is over Borel probability measures on $\Om$.

{\bf (ii)} We have $\dim_H(X_\Om^{(S)}) = \dim_M(X_\Om^{(S)})$ if and only if the tree of prefixes of $\Om$ is spherically
symmetric, i.e.\ for every $k\in \N$, all prefixes of length $k$ have the same (equal) number of continuations in $\Pref_{k+1}(\Om)$.
\end{theorem}

\begin{corollary} \label{cor-dim}
For every set of distinct $n_1,\ldots,n_r\in \N$, with $r\ge 2$, we have $\dim_H(X^{(m)}_{n_1,\ldots,n_r}) < 
\dim_M (X^{(m)}_{n_1,\ldots,n_r})$ where the set $X^{(m)}_{n_1,\ldots,n_r}$ is defined
by (\ref{def-xgen}). 
\end{corollary}

\begin{proof}[Proof (assuming Theorem~\ref{th-var})]
As explained above, $X^{(m)}_{n_1,\ldots,n_r}=X_\Om^{(S)}\subset 
\Sig_m$, where $S$ is the semigroup generated by all prime factors of $n_1,\ldots,
n_r$. Suppose that $n_1<\cdots <n_r$, and let $j\ge 2$ be such that $\ell_j=n_r$. Then every word with symbols
$0,\ldots,m-1$
of length $(j-1)$ is an allowed prefix of $\Om$. Observe that some words of length $(j-1)$
have only one continuation (by 0) to a word of length $j$, namely, those which have nonzeros 
in the positions corresponding to $n_i\in S$, for $i=1,\ldots, r-1$. However,
all other words of length $(j-1)$ (and there will always be some) have $m$ continuations to a word of length $j$.
This shows that the tree $\Pref(\Om)$ is not spherically symmetric and
hence the Hausdorff dimension is strictly less than the Minkowski dimension by Theorem~\ref{th-var}.
\end{proof}

Another way to express the Hausdorff dimension of $X_\Om^{(S)}$ is via a nonlinear system of equations.

\begin{lemma} \label{lem-elem} Let $\Om$ be a closed subset of $\Sig_m$. Then there exists a vector 
$$
\ov{t}=(t(u))_{u\in \Pref(\Om)}\in [1,+\infty)^{\Pref(\Om)},
$$
such that
\be \label{eq-vec1}
t(\varnothing) \in [1,m],\ \ t(u) \in [1, m^{\ell_k(\ell_{k+1}^{-1} +\ell_{k+2}^{-1}+\cdots)}],\ |u|=k,\ k\ge 1,
\ee
\be
\label{eq-elem0}
t(\varnothing)^{\gam(S)} = \sum_{j\in \Pref_1(\Om)} t(j),
\ee
and
\be \label{eq-elem}
t(u)^{\ell_{k+1}/\ell_{k}} = \sum_{j:\ uj\in \Pref_{k+1}(\Om)} t(uj),\ \ \ \forall\ u \in \Pref_k(\Om),\ k\ge 1,
\ee
\end{lemma}

Using $\ov{t}$, it is easy to express the Hausdorff dimension.

\begin{theorem}\label{th-hausd}
We have
$$
\dim_H(X_\Om^{(S)}) = \log_m t(\varnothing),
$$
where $t(\varnothing)$ is from Lemma~\ref{lem-elem}.
\end{theorem}

\section{Preliminaries}

%
 
Let us start with more notations.
Denote
\be \label{def-beta}
\beta_n:= |\{i\le n:\ (i,S)=1\}|.
\ee
We need the following standard fact.

\begin{lemma} \label{lem-betan}
If $p_1,\ldots,p_J$ divide $n$, then
$$
\beta_n = \gam(S)^{-1}n=n\prod_{j=1}^J \Bigl( 1 - \frac{1}{p_j}\Bigr)\,.
$$
\end{lemma}

\begin{proof}
If $n=\prod_{j=1}^J p_j$, then $(i,S)=1$ if and only if $(i,n)=1$, hence $\beta_n=\phi(n)$, Euler's $\phi$-function, for which the formula is
well-known. In the general case, it remains to note that $(i,S)=1$ if and only if $(i+\prod_{j=1}^J p_j,S)=1$.
\end{proof}

Recall that $S = \{\ell_k\}_{k=1}^\infty$, with $1=\ell_1< \ell_2<\ldots$  We will  denote
\be \label{def-bk}
B^{(n)}_k := \{i\in (n/\ell_{k+1}, n/\ell_k]\cap \N:\ (i,S)=1\},
\ee
where $n/\ell_{k+1}$ and $n/\ell_k$ are not necessarily integers.  For every   $n$, let  $K(n)$ be the unique integer such that  
\be
\label{eq1}
\ell_{K(n)}\le n < \ell_{K(n)+1}.
\ee
 Obviously, one has
\be
\label{eqbkn}
\beta_n = \sum_{k= 1}^{K(n)} |B^{(n)}_k| \ \ \mbox{ and } \ \ n = \sum_{k=1}^{K(n)} k |B^{(n)}_k|.
\ee
For a finite word $u$,  we write
$$
u|_{iS}: = u_{i\ell_1} \ldots u_{i\ell_r},\ \ \mbox{where}\ \  i\ell_r \le |u| < i\ell_{r+1},
$$
and for $x=(x_k)_{k\geq 1}\in \Sig_m$ we denote
$$
x_1^n:= x_1\ldots x_n.
$$

We now prove Lemma ~\ref{lem-elem}.

\begin{proof}[Proof of Lemma~\ref{lem-elem}]
Consider the following compact set:
$$
\Xi:= [1,m]\times \prod_{k=1}^\infty [1, m^{\ell_{k}(\ell_{k+1}^{-1} + \ell_{k+2}^{-1} + \cdots)}]^{\Pref_k(\Om)}.
$$
Define a function $\Phi:\, \Xi\to [1,\infty)^{\Pref(\Om)}$ by
$$
\Phi(\ov{t})(u) = {\Bigl(\sum_{uj\in \Pref_{k+1}(\Om)} t(uj)\Bigr)}^{\ell_{k}/\ell_{k+1}},\ \ u\in \Pref_k(\Om),\ k\ge 1,
$$
$$ 
\Phi(\ov{t})(\varnothing) =  {\Bigl(\sum_{j=0}^{m-1} t(j)\Bigr)}^{1/\gam(S)}.
$$
We claim that $\Phi(\Xi)\subset \Xi$. Indeed, if $t(u)\ge 1$ for all $u$, then clearly $\Phi(\ov{t})(u)\ge 1$. For the other inequality, we have,
assuming that $\ov{t}\in \Xi$ and $u\in \Pref_k(\Om),\ k\ge 1$:
$$
\Phi(\ov{t})(u)  \le  m^{(1 + \ell_{k+1}(\ell_{k+2}^{-1}+\ell_{k+3}^{-1} + \cdots))\ell_{k}/\ell_{k+1}}
             =     m^{\ell_{k}(\ell_{k+1}^{-1} + \ell_{k+2}^{-1} + \cdots)},
$$
as desired.
Finally,
$$
\Phi(\ov{t})(\varnothing) \le  m^{(1 + \ell_1(\ell_2^{-1} + \ell_3^{-1} + \cdots))\gam(S)^{-1}}=m^{\ell_1}=m,
$$
by the definition of $\gam(S)$ in (\ref{eq-fs}), and the claim is verified.

Since $\Phi$ is continuous, it has a fixed point by the Tychonov fixed point theorem, which is the desired solution.
(Alternatively, we can start with the vector
of all 1's and iterate $\Phi$. The operator $\Phi$ is monotone in each coordinate, hence there is a coordinate-wise limit, which will be a
fixed point for $\Phi$.)
\end{proof}

Hausdorff dimension will be computed with the help of the following lemma, essentially
due to Billingsley, which we state in the 
symbolic space.

\begin{lemma}[see Proposition 2.2 in \cite{Falc}] \label{lem-Billing}
Let $E$ be a Borel set in $\Sig_m$ and let $\nu$ be a finite Borel measure on $\Sig_m$.

{\bf (i)} If
$\liminf_{n\to \infty} (-\frac{1}{n}) \log_2 \nu[x_1^n] \ge s\ \ \mbox{for $\nu$-a.e.}\ x\in E,$
then $\dim_H(E) \ge s$.

{\bf (ii)} If
$\liminf_{n\to \infty} (-\frac{1}{n}) \log_2 \nu[x_1^n] \le s\ \ \mbox{for all}\ x\in E,$
then $\dim_H(E) \le s$.
\end{lemma}

\section{Minkowski dimension of $X^{(S)}_\Om$}

\begin{proof}[Proof of Theorem~\ref{th-mink}.]
We now compute the Minkowski dimension of $X^{(S)}_\Om$. Recall that
$$\dim_M X^{(S)}_\Om = \lim_{n\to +\infty} \frac {\log_m N_n(X_\Om^{(S)})}{n},$$
where $N_n(X_\Om^{(S)})$ is the number of words of length $n$ that are
prefixes of some $x\in X_\Om^{(S)}$. This holds if the limit exists, one also defines the upper (resp.  lower) dimension $\ov{\dim}_M$ and $\und{\dim}_M$ by taking the liminf (resp.  limsup) instead of the limit.

\smallskip


We need to estimate $N_n(X_\Om^{(S)})$. Fix an integer $r\geq 1$.  Considering integers $n$ of the form  $n=d(\prod_{j=1}^{r+1} \ell_j)( \prod_{i=1}^J p_i)$, for some $d\in \N$, 
 is enough for the
purpose of Minkowski dimension estimates.

By the definition (\ref{def-Xom}), we have $x\in X_\Om^{(S)}$ if and only if $x|_{iS} \in \Om$ for all $i,\ (i,S)=1$. It follows that $x_1^n$ is
a beginning (prefix) of some $x\in X_\Om^{(S)}$ if and only if 
$$
x_1^n|_{iS}\in \Pref_k(\Om),\ \ \forall\,i\in B_k^{(n)},\ k=1,\ldots,K(n),
$$
where $B_k^{(n)}$ is defined in (\ref{def-bk}).
Thus, using the definition \eqref{defak} of $A_k$, we obtain
$$
N_n(X_\Om^{(S)})= \prod_{k=1}^{K(n)} A_k^{|B^{(n)}_k|}\,.
$$

By  the choice of $n$,   $n/\ell_k$ and $n/\ell_{k+1}$ are integers  for every $k\leq r$. By Lemma~\ref{lem-betan}, one sees that 
\be \label{eq-bk}
|B^{(n)}_k|=  \beta_{\frac{n}{l\ell_k}}- \beta_{\frac{n}{\ell_{k+1}}}=\gam(S)^{-1} n \Bigl(\frac{1}{\ell_k} - \frac{1}{\ell_{k+1}}\Bigr)\ \ \mbox{for}\ k\le r,
\ee
hence, 
$$
\frac{1}{n}\log_m N_n(X_\Om^{(S)})\ge \sum_{k=1}^r  \frac{|B^{(n)}_k|}{n} \log_m A_k=
\gam(S)^{-1} \sum_{k=1}^r \log_m(A_k)\Bigl(\frac{1}{\ell_k} - \frac{1}{\ell_{k+1}}\Bigr).
$$
We obtain that 
\be \label{mink-lo}
\underline{\dim}_M(X_\Om^{(S)}) \ge\gamma(S)^{-1} \sum_{k=1}^r \log_m(A_k)\Bigl(\frac{1}{\ell_k} - \frac{1}{\ell_{k+1}}\Bigr).
\ee
On the other hand,  for $r+1\leq k \leq K(n)$,  $A_k$ is bounded from above by  $m^k$. This yields
\be \label{eq-bik}
N_n(X_\Om^{(S)})\le \prod_{k=1}^r  A_k^{|B^{(n)}_k|}\cdot \prod_{k=r+1}^{K(n)}m^{  k |B_k^{(n)}|} = \prod_{k=1}^r  A_k^{|B^{(n)}_k|}\cdot m^{n - \sum_{k=1}^r k |B_k^{(n)}|} \,,
\ee
where \eqref{eqbkn} has been used.
We have
$$
\sum_{k=1}^r k |B_k^{(n)}|=\gamma(S)^{-1}n \sum_{k=1}^r\Bigl(\frac{k}{\ell_k} - \frac{k}{\ell_{k+1}}\Bigr)
= \gamma(S)^{-1}n \Bigl( - \frac{r}{\ell_{r+1}}+\sum_{k=1}^r \frac{1}{\ell_k}\Bigr).
$$
Thus,
$$
n - \sum_{k=1}^r k |B_k^{(n)}| = \gamma(S)^{-1}n\Bigl(\frac{r}{\ell_{r+1}}+\sum_{i=r+1}^\infty \frac{1}{\ell_i}\Bigr).
$$
It follows from (\ref{eq-bik}), again using (\ref{eq-bk}), that
$$
\ov{\dim}_M(X_\Om^{(S)}) \le \gamma(S)^{-1} \left ( \, \sum_{k=1}^r \log_m(A_k)\Bigl(\frac{1}{\ell_k} - \frac{1}{\ell_{k+1}}\Bigr)+
 \Bigl(\frac{r}{\ell_{r+1}}+\sum_{i=r+1}^\infty \frac{1}{\ell_i}\Bigr) \, \right ),
$$
and letting $r\to \infty$ here and in (\ref{mink-lo}) yields the desired formula.
\end{proof}

\section{Lower bound for  $\dim_H X^{(S)}_\Om$ in Theorem~\ref{th-var}(i)}

Given a probability measure $\mu$ on $\Om$, we set
\be\label{eq-proba}
\Pmu[u] :=\prod_{i\le |u|,\, (i,S)=1} \mu[u|_{iS}].
\ee
It follows from (\ref{def-Xom}) that $\Pmu$ extends to a Borel probability measure supported on $X_\Om^{(S)}$.

\medskip
 
Fix a probability measure $\mu$ on $\Om$.  Recall that  $\alpha_k$ is the partition of $\Om$ into cylinders of length $k$. We are going to demonstrate that for every $r\in \N$,
\be \label{eq-lb2}
\liminf_{n\to \infty} \frac{-\log_m \Pmu[x_1^n]}{n} \ge
\gam(S)^{-1} \sum_{k=1}^r H^\mu(\alpha_k)\Bigl(\frac{1}{\ell_k} - \frac{1}{\ell_{k+1}}\Bigr)\ \ \mbox{for $\Pmu$-a.e.}\ x.
\ee
Then, letting $r\to \infty$ will yield $\dim_H(X_\Om^{(S)})\ge s(\Om,\mu)$ by Billingsley's Lemma~\ref{lem-Billing}(ii). 

\smallskip

Let us fix an integer $r\geq 1$. Again, to verify (\ref{eq-lb2}), we can restrict ourselves to the integers of the form 
 \be
\label{defn}
n=d \cdot  \Bigl(\, \prod_{j=1}^{r+1} \ell_j\,\Bigr) \cdot \Bigl(\,  \prod_{i=1}^J p_i \,\Bigr ),  \ \ d\in \N.
\ee

In view of (\ref{eq-proba}), we have
\be
\label{eq2}
\Pmu[x_1^n]  \  =  \  \prod_{k=1}^{K(n)}  \prod_{i \in B^{(k)}_n} \mu[x^n_1|_{iS}]  , 
\ee
hence, when $n$ is large enough, 
\be \label{eq-lb3}
\Pmu[x_1^n]   \ 
   \le \  \prod_{k=1}^r \ \prod_{i\in B^{(n)}_k} 
\mu[x_1^n|_{iS}].
        \ee
        Note that $x_1^n|_{iS}$ is a word of length $k$ for $i\in B^{(n)}_k$, which is
        a beginning of a sequence in $\Om$.
        Thus, $[x_1^n|_{iS}]$ is an element of the partition $\alpha_k$.
        
        The random variables $x\mapsto -\log_m \mu[x_1^n|_{iS}]$ are i.i.d. 
        for $i\in B_k^{(n)}$ (hence, there are $|B_k^{(n)}|$ of them), 
            and their expectation  equals $H^\mu(\alpha_k)$, by the definition of entropy.
                    Fixing $k $ with $k\le r$ and taking  $n$ of the form \eqref{defn}, since $|B_k^{(n)}|$ tends to infinity as $d$ tends to infinity, we get an infinite sequence of i.i.d.\ random variables. Therefore, by a version of the Law of Large
                    Numbers,
            \be \label{eq-lb4}
\mbox{ for $\Pmu$-a.e.\ $x$, } \ \forall\ k\le r,\ \ \   \frac{1}{| B_k^{(n)}|} \sum_{i\in B_k^{(n)}}
       -\log_m \mu[x_1^n|_{iS}] \  \longrightarrow  \ H^\mu(\alpha_k)
       \ee
       as $n $ (and thus $d$) tends to infinity.
      Using (\ref{eq-lb3}), we deduce that
       $$ 
       \frac{-\log_m \Pmu[x_1^n]}{n} \ge  \frac 1 n  \sum_{k=1}^r \sum_{i\in B^{(n)}_k}   -\log_m \mu[x_1^n|_{iS}].  $$
      In view of (\ref{eq-bk}),
       we obtain
   $$ 
      \frac{-\log_m \Pmu[x_1^n]}{n} \ge  \sum_{k=1}^r  \gam(S)^{-1}  \Bigl(\frac{1}{\ell_k} - \frac{1}{\ell_{k+1}}\Bigr)   \sum_{i\in B^{(n)}_k}    \frac{ -\log_m \mu[x_1^n|_{iS}]}{| B_k^{(n)}|} . $$
      Taking the liminf as $n$ tends to infinity and using (\ref{eq-lb4}),   we confirm (\ref{eq-lb2})  for $\Pmu$-a.e.\ $x$, completing the proof.

\section{Upper bound for  $\dim_H X^{(S)}_\Om$, and Theorem ~\ref{th-hausd}}

To find the upper bound, we will construct an explicit measure on  $X^{(S)}_\Om$ which has the right dimension. Since we will be able to compute the Hausdorff dimension of this measure $\mu$ (it will be $\log_m\, t_{\emptyset}$), this will allow us to conclude.

\medskip

In view of Lemma~\ref{lem-elem}, we can define the probability measure $\mu$ on $\Om$ such that
\begin{eqnarray}
 \label{def-meas1}
 && \mu[j]:=\frac{t(j)}{t(\varnothing)^{\gam(S)}}\ \ \mbox{for all}\ j\in \Pref_1(\Om),\\
 \label{def-meas2}
\  \ \  \  \ \ \ \forall \, k\ge 1,& & \mu[uj|u]:= \frac{t(uj)}{t(u)^{\ell_{k+1}/\ell_k}}\ \ \mbox{for all}\ u\in \Pref_k(\Om),\ uj\in \Pref_{k+1}(\Om).
\end{eqnarray}
Thus, for every $u=u_1\ldots u_k\in \Pref_k(\Om)$, 
\be \label{def-meas3}
\mu[u]= t(\varnothing)^{-\gam(S)} t(u_1)^{1-\frac{\ell_2}{\ell_1}} \cdots t(u_1\ldots u_{k-1})^{1-\frac{\ell_k}{\ell_{k-1}}} t(u).
\ee

 We are going to use Billingsley's Lemma~\ref{lem-Billing}(i), for which we need to estimate the $\liminf_{n\to \infty}
\frac{-\log_m\Pmu [x_1^n]}{n}$ from above, {\bf for every $x\in X_\Omega^{(S)}$}. 

We will assume throughout the proof that $p_1,\ldots,p_J$ divide $n$.

 Recalling \eqref{eq2}, we need to estimate $\prod_{i\in B_k^{(n)}} \mu[x_1^n|_{iS}] $. By (\ref{def-meas3}), we have
 \begin{eqnarray*}
& &  \prod_{i\in B_k^{(n)}} \mu[x_1^n|_{iS}]    =   \prod_{i\in B_k^{(n)}} \mu[x_{i\ell_1} x_{i\ell_2}\ldots x_{i\ell_k}]\\
  & = &  \prod_{i\in B_k^{(n)}} t(\varnothing)^{-\gam(S)} t(x_{i\ell_1})^{1-\frac{\ell_2}{\ell_1}}  
t(x_{i\ell_1} x_{i\ell_2})^{1-\frac{\ell_3}{\ell_2}} \cdots t(x_{i\ell_1} x_{i\ell_2}\ldots x_{i\ell_{k-1}})^{1-\frac{\ell_k}{\ell_{k-1}}}\\
&&\ \ \ \ \ \ \ 
  \times  t(x_{i\ell_1} x_{i\ell_2}\ldots x_{i\ell_{k}})\\
   & = & t(\varnothing)^{-|B_k^{(n)}|\gam(S)}  \prod_{i\in B_k^{(n)}}    t(x_{i\ell_1} x_{i\ell_2}\ldots x_{i\ell_{k}})     \prod_{k'=1}^{k-1}  t(x_{i\ell_1} x_{i\ell_2}\ldots x_{i\ell_{k'}})^{1-\frac{\ell_{k'+1}}{\ell_{k'}}} \,.
 \end{eqnarray*}
    
    Using   (\ref{eqbkn}), the product can be rewritten as
  \begin{eqnarray*}
\Pmu[x_1^n]   & =  &   t(\varnothing)^{-\beta_n \gam(S)}  \times  \left( \prod_{k=1}^{K(n)}    \prod_{i\in B_k^{(n)}}   \prod_{k'=1}^{k} t(x_{i\ell_1} x_{i\ell_2}\ldots x_{i\ell_{k'}}) \right) \\
&    \times & \left( \prod_{k=1}^{K(n)}    \prod_{i\in B_k^{(n)}}   \prod_{k'=1}^{k-1} t(x_{i\ell_1} x_{i\ell_2}\ldots x_{i\ell_{k'}})  ^{1-\frac{\ell_{k'+1}}{\ell_{k'}}} \right)  .
    \end{eqnarray*}
Observe that if $k$ is given in $\{1, \ldots,K(n)\}$ and    
$(i,S)=1$ with $i\leq n/\ell_k$, then the term $t(x_{i\ell_1} x_{i\ell_2}\ldots x_{i\ell_{k}})$ appears exactly once in the first product above. 
Similarly,  if $k$ is given in $\{1,\ldots,K(n)-1\}$ and    
$(i,S)=1$ with $i\leq n/\ell_{k+1}$, then the term $t(x_{i\ell_1} x_{i\ell_2}\ldots x_{i\ell_{k}})^{1-\frac{\ell_{k+1}}{\ell_{k}}}$ 
appears  once in the second product above. We deduce that
\begin{eqnarray}
 \Pmu[x_1^n] & = &  t(\varnothing)^{-\beta_n \gam(S)} \prod_{k=1}^{K(n)} \prod_{\stackrel{i\le n/\ell_k}{(i,S)=1}}
                        t(x_{i\ell_1} x_{i\ell_2}\ldots x_{i\ell_k}) \nonumber \\
                             & \times & \prod_{k=1}^{K(n)-1} 
                             \prod_{\stackrel{i\le n/\ell_{k+1}}{(i,S)=1}}
                        t(x_{i\ell_1} x_{i\ell_2}\ldots x_{i \ell_k})^{-\ell_{k+1}/\ell_k}.\label{eq-pmu1}
\end{eqnarray}
Therefore, taking Lemma~\ref{lem-betan} into account, we have
\begin{eqnarray}
\frac{-\log_m\Pmu [x_1^n]}{n} & = &  \log_m t(\varnothing) + \sum_{k=1}^{K(n)-1}
  \sum_{\stackrel{i\le n/\ell_{k+1}}{(i,S)=1}} \frac{\ell_k^{-1} \log_m t(x_{i\ell_1} x_{i\ell_2}\ldots x_{i\ell_k})}{n/\ell_{k+1}} 
\nonumber \\
                              & - & \sum_{k=1}^{K(n)}
         \sum_{\stackrel{i\le n/\ell_k}{(i,S)=1}} \frac{\ell_k^{-1} \log_m t(x_{i\ell_1} x_{i\ell_2}\ldots x_{i\ell_k})}{n/\ell_k}\,.
\label{eq-pmu2}
\end{eqnarray}
Denote
\be \label{def-av}
\Av_k(x,s):= \sum_{\stackrel{i\le s}{(i,S)=1}} \frac{\ell_k^{-1} \log_m t(x_{i\ell_1} x_{i\ell_2}\ldots x_{i\ell_k})}{s}\,,
\ee
where $s>0$ is not necessarily in $\N$. Then (\ref{eq-pmu2}) becomes
\be \label{eq-pmu3}
\frac{-\log_m\Pmu [x_1^n]}{n} = \log_m t(\varnothing) + \sum_{k=1}^{K(n)-1}
                            \Av_k\Bigl(x, \frac{n}{\ell_{k+1}}\Bigr) -  \sum_{k=1}^{K(n)}  \Av_k\Bigl(x, \frac{n}{\ell_k}\Bigr). 
\ee

Next, observe that by Lemma~\ref{lem-elem} and the fact that  $ t(u) \in [1, m^{\ell_k(\ell_{k+1}^{-1} +\ell_{k+2}^{-1}+\cdots)}]$ if $|u|=k$, we have
\be \label{eq-esta}
\Av_k(x,s)\le  \sum_{\stackrel{i\le s}{(i,S)=1}} \frac{ (\ell_{k+1}^{-1} +\ell_{k+2}^{-1}+\cdots)  }{s}  \leq \sum_{i=k+1}^\infty \ell_i^{-1}\ \ \ \mbox{for all}\ s>0,
\ee
hence for $r< K(n)$ we have
\begin{eqnarray} 
\sum_{k=r+1}^{K(n)}  \Av_k\Bigl(x, \frac{n}{\ell_k}\Bigr) &  < & \sum_{k=r+1}^\infty \sum_{i=k+1}^\infty \ell_i^{-1} \nonumber \\
                                                      &  < & \sum_{i=r+2}^\infty i \ell_i^{-1}=: \Ek_{r+2}\to 0,\ \ \mbox{as}\ 
r\to \infty. \label{eq-Av2}
\end{eqnarray}
The convergence $\sum_{i=1}^\infty i\ell_i^{-1}<\infty$ is clear, since $\ell_i$ grows faster than any polynomial. The same holds for $\sum_{k=r+1}^{K(n)}  \Av_k\Bigl(x, \frac{n}{\ell_{k+1}}\Bigr) $.

\medskip

 We will estimate from above the averages
$$
\Ak(x,n,M):= \frac{1}{M} \sum_{j=1}^M \frac{-\log_m \Pmu[x_1^{n\ell_j}]}{n\ell_j}.
$$

 \begin{lemma}
 \label{lem-liminf}
For all $x\in X_\Om^{(S)}$, $\liminf_{M\to \infty} \Ak(x,n,M)  \le \log_m t(\varnothing)    $.
\end{lemma}

\begin{proof}
Fix $\eps>0$ and choose $r\in \N$ such that $\Ek_{r+2}<\eps$. Choose $n\in \N$  of the form \eqref{defn}. 
   By  (\ref{eq-pmu3})  and (\ref{eq-Av2}),
\be \label{eq-sum}
\Ak(x,n,M)  \le \log_m t(\varnothing) + \frac{1}{M}\sum_{k=1}^r \sum_{j=1}^M \Bigl( \Av_k\bigl(x, \frac{n\ell_j}{\ell_{k+1}}\bigr) -
\Av_k\bigl(x, \frac{n\ell_j}{\ell_{k}}\bigr) \Bigr) + 2\eps.
\ee

We are going to argue that the majority of the terms in each of the interior sums  above cancel out. 
In fact,  when $\ell_{k+1}$ divides $\ell_j$, there exists a unique $j'\ < j$ such that $\frac{\ell_{j'} }{\ell_k} = \frac{\ell_j}{\ell_{k+1}}$,
hence the term   $ \Av_k\bigl(x, \frac{n\ell_j}{\ell_{k+1}}\bigr)$ cancels out with the term $\Av_k\bigl(x, \frac{n\ell_{j'}}{\ell_{k}}\bigr)$. 
For this to happen, all we need is that $\ell_{k+1}$ divide $\ell_j$.

We will show that this occurs for most of the terms in the sum above.
\begin{lemma} \label{lem-divide}
For any $r\in \N$, let $ {F  _M}=\{j \le M:\ \ell_{k+1} | \ell_j,\ \forall\, k=1,\ldots,r\}$.  Then
\be \label{eq-divide}
\lim_{M\to \infty} M^{-1} |F_M|=1.
\ee
\end{lemma}

\begin{proof}
We have
$$
S = \left\{p_1^{m_1}\cdot\ldots\cdot p_J^{m_J}:\ m_1,\ldots, m_J\in \N\cup \{0\}\right\}.
$$
Let $C=C(r)\in \N$ be such that $\ell_2,\ldots,\ell_{r+1}$ all divide $\prod_{j=1}^J p_j^C$. Let us define 
\begin{eqnarray*}
 E_s & =  &  \left\{(m_1,\ldots,m_J)\in \Z_+^J:\ \sum_{j=1}^J m_j\log p_j \le \log( s) \right\}\\
 \widetilde E_s  & =  &  \left\{(m_1,\ldots,m_J)\in 
\Z_+^J:\ \min_j m_j \ge C,\ \sum_{m=1}^J m_j\log p_j \le \log (s) \right\}.
\end{eqnarray*}
For $s=\ell_M$, we have $|E_{\ell_M}| = M$ and 
$$  \widetilde E_{\ell_M} \subset \Big\{(m_1,\ldots,m_J)\in \Z_+^J: \  \ell:=p_1^{m_1}\cdots p_J^{m_J} \leq \ell_M  \mbox{ and } \ell_2,..., \ell_{r+1} \mbox{ $|$ } \ell  \Big\} \subset E_{\ell_M},$$
from which we deduce that $| \widetilde E_{\ell_M} | \leq | F_M | \leq M=|E_{\ell_M}|$.
Then (\ref{eq-divide}) follows from
$$
\lim_{s\to +\infty} \frac{|E_s| }{| \widetilde E_s|}=1,
$$
which is clear.
\end{proof}

 From Lemma~\ref{lem-divide}, we choose $M>0$ so large that 
$M(1-\eps ) \leq |F_M| \leq M .$
 As said above, when $j\in F_M$,   the term $ \Av_k\bigl(x, \frac{n\ell_j}{\ell_{k+1}}\bigr)$ cancels out with some term  $
\Av_k\bigl(x, \frac{n\ell_{j'}}{\ell_{k}}\bigr) \Bigr)$  in \eqref{eq-sum}. 
The cardinality of the remaining terms is thus less than $2M\eps$, 
and each remaining  term is upper bounded  above by $\gam(S)$, see (\ref{eq-esta}), hence (\ref{eq-sum}) yields
$$
\liminf_{M\to \infty} \Ak(x,n,M)  \le \log_m t(\varnothing) +   2\eps \gamma(S)+ 2 \eps. 
$$
for all $x\in X_\Om^{(S)}$. Letting $\eps$ go to zero yields the result. 
\end{proof}

From   Lemma \ref{lem-liminf}, we conclude that  $\liminf_{j\to \infty} \frac{-\log_m\Pmu [x_1^j]}{j}\le \log_m t(\varnothing)  $
for all $x\in X_\Om^{(S)}$. It is key that this inequality holds true for all $x\in X_\Om^{(S)}$, not only for almost all  $x\in X_\Om^{(S)}$. Now we conclude by Billingsley's Lemma~\ref{lem-Billing}(i) 
that $\dim_H X^{(S)}_\Omega \leq \log_m t(\varnothing)$, and the upper bound in Theorem~\ref{th-hausd} is proved.

\medskip

\begin{sloppypar}
\noindent {\em Proof of the lower bound in Theorem~\ref{th-hausd}.} We  deduce the lower bound 
$\dim_H(X_\Om^{(S)}) \ge \log_m t(\varnothing)$ from the lower bound in Theorem~\ref{th-var}(i) and the following lemma, which asserts that the measure constructed in  \eqref{def-meas1} and \eqref{def-meas2} is ``optimal".
\end{sloppypar}

\begin{lemma} \label{lem-comsom}
The measure $\mu$ on $\Om$ defined by (\ref{def-meas1})  and \eqref{def-meas2}  satisfies
\be \label{eq-meas21}
s(\Om,\mu) = \log_m t(\varnothing).
\ee
\end{lemma}

\begin{proof} We have by (\ref{def-meas1}),
\begin{eqnarray*}
H^\mu(\alpha_1) & = & - \sum_{j=0}^{m-1} \frac{t(j)}{t(\varnothing)^{\gamma(S)}} \log_m\bigl( \frac{t(j)}{t(\varnothing)^{\gamma(S)}}
                         \bigr) \\
                        & = & \gamma(S) \log_m t(\varnothing) - \sum_{j=0}^{m-1} \frac{t(j)}{t(\varnothing)^{\gamma(S)}} \log_m t(j)
                 \\       &  = & \gam(S) \log_m t(\varnothing) - \sum_{j=0}^{m-1} \mu[j]\log_m t(j).
\end{eqnarray*}
Further,
\begin{eqnarray*}
H^\mu(\alpha_{k+1}|\alpha_k) & = & \sum_{[u]\in \alpha_k} \mu[u] \Bigl(- \sum_{j:\,[uj]\in \alpha_{k+1}}
\frac{t(uj)}{t(u)^{\ell_{k+1}/\ell_k}} \log_m \frac{t(uj)}{t(u)^{\ell_{k+1}/\ell_k}}\Bigr) \\
     & = & \sum_{[u]\in \alpha_k} \mu[u] \Bigl(\frac{\ell_{k+1}}{\ell_k} \log_m t(u)- \sum_{j:\,[uj]\in \alpha_{k+1}}
                                \mu[uj|u] \log_m t(uj)\Bigr)\\
                                & = & \frac{\ell_{k+1}}{\ell_k} \sum_{[u]\in \alpha_k} \mu[u] \log_m t(u) - \sum_{[v]\in \alpha_{k+1}}
                                \mu[v]\log_m t(v).
\end{eqnarray*}
Thus,
$$
\frac{H^\mu(\alpha_{k+1}|\alpha_k)}{\ell_{k+1}} = \frac{1}{\ell_k} \sum_{[u]\in \alpha_k} \mu[u] \log_m t(u) 
                                                 - \frac{1}{\ell_{k+1}} \sum_{[v]\in \alpha_{k+1}} \mu[v]\log_m t(v).
$$
Now it is clear that the sum in (\ref{def-som}) ``telescopes''. Note also that 
$$
\frac{1}{\ell_k} \sum_{[u]\in \alpha_k} \mu[u] \log_m t(u)\le \sum_{i=k+1}^\infty \ell_i^{-1} \to 0,\ \ \mbox{as}\ k\to \infty
$$
by Lemma~\ref{lem-elem}. It follows that
$s(\Om,\mu)=\log_m t_\varnothing$, as desired. This completes the proof of the lemma, and of Theorem~\ref{th-hausd}.
\end{proof}

\smallskip

All that remains to prove is the part (ii) of Theorem \ref{th-var}.

\smallskip

\noindent {\em Proof of Theorem~\ref{th-var}.}(ii) Every term in the expression for $\dim_M(X_\Om^{(S)})$ in Theorem~\ref{th-mink} dominates the corresponding term in (\ref{def-som}),
with equality if and only if $\mu$ assigns the same measure to each cylinder of length $k$, for every $k$. This is true for the
``natural'' uniform measure, when $\Pref(\Om)$ is spherically symmetric, and cannot hold otherwise. \qed

\section{Numerics and further examples}

In this section we  introduce a ``geometric'' argument used to determine the Minkowski dimension of several examples, and which allowed us to write an algorithm to produce the values of $A_k$ for large values of $k$. The main idea is that we use a triangular arrangement of the sets $iS$, for  $(i,S)=1$.

Let $p < q\in\mathbb{N}$; $(p,q)=1$ and $S=\langle p, q\rangle$ be the semigroup generated by $p,q$. Let $F\colon\Sigma_2\times\Sigma_2\times\Sigma_2\to\mathbb{R}$ be a function depending only on the first coordinates, i.e. $F(x,y,z)=F(x_1,y_1,z_1)$. We are interested in the sets
\begin{equation}
\label{defF}
X_F:=\left\{x\in\Sigma_2\, :\, F(x_\ell,x_{p\ell},x_{q\ell})=0 \, \  \forall \,  \ell\in\mathbb{N}\right\}.
\end{equation}

In $\mathbb{N}^2$ we consider the  infinite triangular  matrix 
\[
\Delta:=\left\{p^{n-m} q^m\, :\, n\ge m\ge 0 \right\}.
\]
 For instance, when $p=2$ and $q=3$, the matrix $\Delta$ (and $S$ itself) can be represented as 
$$\begin{array}{ccccccccccc}
  &&&&&&& 2187  & ... \\
  &&&&&& 729 & 1458 & 3916  & ...    \\ 
   &&&&& 243 & 686 & 1378 & 2756  & ...    \\
  &&&& 81 & 162 & 324 & 648 & 1296  & ...   \\
  &  & & 27 & 54 & 108 & 216 & 432  & 864  & ...   \\
    &&9&18&36&72&144  & 288 & 576  & ...   \\ 
  & {3}&6&12&24&48&96&192&384 & ...   \\ 
     {1}& {2}&4&8&{16}&32&64&128&256 & ...  
\end{array}  $$
To determine the Minkowski dimension we need to consider the truncated sectors
\[
\Delta_N^i:=\{ p^{n-m} q^m\in\Delta\, :\, ip^{n-m} q^m\le N\},\qquad N\in\mathbb{N}. 
\]
The right-hand-side boundary of this sector is approximately determined by a grid approximation of a line with slope 
\[
\gamma=\frac{\log p}{\log p -\log q}<0.
\]
This ``broken'' line is determined by the Sturmian sequence associated to $\gamma$, as follows. Given an irrational slope $\gamma$ we consider the line $\gamma x$. It will intersect the integer grid consecutively (starting at the origin) in horizontal and vertical segments. It is classical that if $|\gamma|>1$, then for some integer $n_\gamma$ depending on $\gamma$ only,  a sequence of $n_\gamma$ or $n_\gamma+1$ vertical intersections will be followed by a single horizontal intersection. If $|\gamma|<1$ then a sequence of $n_\gamma$ or $n_\gamma+1$ horizontal directions will be followed by a single vertical intersection. The set of  boundary ``squares'' $\partial_r^i(N)$ of the configuration in the truncated triangle are the integer grid squares that have an intersection with the line $\gamma x + N$. These squares are given by
\[
\partial_r^i(N)=\{(n,m)\, :\, ip^{n-m}q^m\le N<ip^{n-m+1}q^m\}.
\]
Denote by $n(i,N)$ the maximum of those integers $n$ for which one can find a pair $(n,m) \in \partial_r^i(N) $. The numbers $m(i,N)$
are defined similarly.

If $p^2>q$, i.e. $|\gamma|>1$ (for instance when $p=2$ and $q=3$), then  they have the structure 
\[
(n_1,m_1),(n_1,m_1+1)\cdots, (n_1,m_1+k_1), (n_2,m_1+k_1),\cdots,(n_2,m_2),\cdots,
\]
 where $n_1=n(i,N)+1$, $m_1=0$, and for $j\geq 2$, $ n_j=n_{j-1}-1$, $ m_j=m_{j-1}+k_{j-1}$ (with $k_{j-1}=n_\gamma$ or $ n_\gamma+1$), until $m_j= n_j=m(i,N)+1$.   The case where $p^2<q$ (i.e. $|\gamma|<1$) is symmetric. 

So we have exactly $(n(i,N)+1)-(m(i,N)+1) = n(i,N)-m(i,N)$ ``steps'' to the left and $m(i,N)+1$ steps up. 
All together we have $n(i,N)+1$ boundary ``squares''. The ``squares'' $(n_j-1,m_j+k_j)\in\partial_r^i(N)$ are called {\em pivotal} if $(n_j,m_j+k_j)\in\partial_r^i(N)$. There are exactly $n(i,N)-m(i,N)$ such squares.

For instance, when $p=2$, $q=3$, each truncated section is a ``triangle" of right slope $-\log 2/(\log 3-\log 2)$. When $N=27$ and $N=243$, the truncated triangles are
 $$\begin{array}{cccccccccccccccccccc}
    &&&&&&&&&&&&& & 243\\
    &&&&&&&&&&&&& 81 &162 \\
     &  & & 27    &&&&    && &  & & 27 & 54 &108&216 \\
    &&9&18     &&&&   &&&&9&18&36&{\bf 72} &144\\ 
  & {3}&6&12&24   & &&&&     & {3}&6&12&{\bf 24}&{\bf 48}&96&192  \\ 
     {1}& {2}&4&8&{16}    &&&&&       {1}& {2}&4&8&{16}&32&64&128 
\end{array}  $$
 
\begin{center}
Truncated triangles $\Delta^i_N$  for $i=1$, $N=27 $ and $N=243$. In bold characters,\\
the three integers $(24, 2\times 24, 3\times 24)$ located in a ``corner".
\end{center}

Using this representation, one observes that the conditions on the three digits   $(x_\ell,x_{p\ell},x_{q\ell})$ in the sets $X_F$  from \eqref{defF} are then expressed in terms of three consecutive terms located in a ``corner", since the integers $(n,m)$ corresponding to  $(\ell,p\ell,q\ell)$ always have the same relative locations inside the truncated sections $\Delta^i_N$ (see the example above).

Let us return to our main example: $S = \langle 2, 3 \rangle$,   $\gamma(S)=3$   with
 \begin{eqnarray*}
 X_{2,3} & =  & \left\{{(x_k)}_{k=1}^\infty \in \Sig_2:\ x_k x_{2k} x_{3k}=0\ \forall\,k\in \N\right\}.\\
\Om & = & \left\{{(\om_k)}_{k=1}^\infty  \in \Sig_2:\ \om_i \om_j \om_k=0\ \mbox{if}\ 2\ell_i=\ell_j,\ 3 \ell_i = \ell_k\right\}.
\end{eqnarray*}
 The table below lists the first elements of $S$, denoted $\ell_k$, and the corresponding 
$A_k=|\Pref_k(\Om)|$,  needed to compute the
Minkowski dimension, see Theorem~\ref{th-mink}. 

\smallskip

\begin{eqnarray*}\begin{array}{|c|c|c|c|c|c|c|c|c|c|c| c|c|c|c|c|}\hline
 k & 1  & 2 &  3  & 4  & 5 & 6 & 7 &8 &9 &10 & 11 & 12& 13   & 14 & 15 \\
 \hline 
 \ell_k & 1 & 2 & 3 & 4& 6 & 8 & 9 & 12 & 16 & 18   & 24 & 27& 32 & 36& 48 \\
 \hline
  A_k & 2 & 4 & 7 & 14 & 25& 50 & 90 & 160 & 320 & 584 & 1039 & 1861& 3722 & 6772& 12050 \\
  \hline 
\end{array} \\
 \begin{array}{|c|c|c|c|c|c|c|c|c|c|c| c|c| c|}\hline
 k   & 16 & 17 &18 &19  & 20 & 21   & 22& 23 & 24 \\
 \hline 
 \ell_k   & 54& 64 & 72 & 81  & 96  & 108 & 128 & 144 & 162 \\
 \hline
  A_k  & 21909 & 43818 & 79784 & 143028   & 254433& 461161 & 922322& 1679220 & 3055130\\
  \hline 
\end{array}\\
\begin{array}{|c|c|c|c|c|c|c|c|c|c|c| c|c| c|}\hline
 k & 25 & 26& 27   & 28  & 29 & 30 & 31   \\
 \hline 
 \ell_k &  192 & 216& 243 & 256 & 288 & 324 &  384       \\
 \hline
  A_k& 5434757 & 9855663 & 17665509 & 35331018 & 64326200 & 116676724 &  207555865  \\
  \hline 
\end{array}
\end{eqnarray*}

 \medskip

The algorithm to determine $\Pref_{k}(\Om)$ is based on  the ``triangle" representation.  Consider an integer $k$ and  $\ell_k$, and then the corresponding truncated triangle $\Delta^1_{\ell_k}$ (for instance for $k=12$ (equivalently, for  $N=27$) as in the figure above). A word in $\Pref_{k}(\Om)$ can be represented as an array of zeros and ones, whose entries are located at the grid vertices in the truncated triangle. Moreover, the condition $x_k x_{2k} x_{3k}=0$ means that there are no three consecutive ones in a ``corner".
 $$\begin{array}{ccccccccccccccccccccccccccc}
     &  & & 27    &&&    && &  & & 1&&&    && &  & & 1 \\
    &&9&18     &&& \longrightarrow  &&&&0&{\bf 1}&     &&&    &&&0&1 \\ 
  & {3}&6&12&24   & &&&     & {0  }    & {\bf 1}   & {\bf 1}  & 1 &   &    &&&   {0  }& 1   & 0  & 1  \\ 
     {1}& {2}&4&8&{16}    &&&&       {1}& {0}&1&0&{1}  &&&&       {1}& {0}&1&0&{1}
     \end{array}  $$
\begin{center}
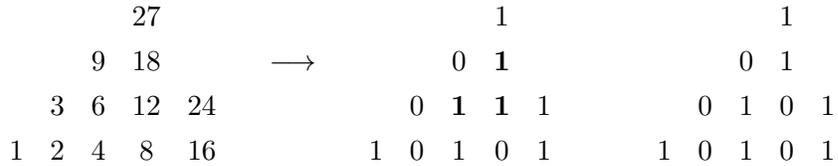
Truncated triangle with   $k=12 $, one bad and one good configuration.  \end{center}
\smallskip
\smallskip
In the algorithm, we first generate all possible first two lines of 0-1 bits (which yields $2^9$ configurations when $k=12$) and then keep only those which are admissible (the key point is the simplicity of the selection procedure, i.e.\ no three ones in a corner). Then we generate the third line, and keep only the suitable configurations. And so on. This argument yields the values of $A_k$ in a very reasonable time, up to $k=31$.

Actually it is a very interesting combinatorial problem in itself to determine the number of admissible configurations $A_k$, even in a simpler geometrical context (for instance in a $N\times N$-square, with the forbidden ``corner"). There is numerical evidence that $A_k$ has a power law of $k$, but to confirm this would certainly require further investigations.

\medskip

It is not hard to show that
$ 7/4 \leq A_{k+1}/A_k \le 2$. Using this and the data in the table, we obtain for the Minkowski dimension (Theorem \ref{th-mink})
$$
0.9573399 \le \dim_M(X_{2,3}) \le 0.9623350.
$$

To estimate the Hausdorff dimension we can use Theorem~\ref{th-hausd}. We get explicit rigorous estimates by going to a fixed level $n$, and
either assuming that there are no restrictions further on, whence $t(u)=2^{\ell_n(\ell_{n+1}^{-1} + \cdots)},\ |u|=n$, to get an upper bound, or
to assume that all the digits that follow are 0's, whence $t(u)=1$, to get a lower bound. Then all the values of $t(u)$, with $|u|\leq n-1$, are obtained recursively  using \eqref{eq-elem} and \eqref{eq-elem0}.
We did the calculation with $n=25$ and obtained
$$
0.9246585  < \dim_H(X_{2,3}) < 0.9405728
$$
(the convergence is slow, but we think that the upper bound is closer to the truth).

\begin{sloppypar}
\subsection{Further examples}
The difficulty with the set $X_{2,3}$ and the function $F(x,y,z)=x_1y_1z_1$  comes essentially from the fact that,  fixing the bits at the frontier of the truncated triangles is not enough  to deduce the values of all digits inside the triangle (there is some long-range dependence between the bits). For some suitable functions $F$ and the corresponding sets $X_F$, this is not the case, and the situation is easier.
\end{sloppypar}

\medskip

\noindent {\bf Definition.} {\em
We call a function $F$ {\em deterministic} if for all $i,j\in\{0,1\}$ there is a unique solution $k\in\{0,1\}$ for one of the following implicit equations
$$i)  \ F(i,j,k)=0,  \ \  \ \ ii) \ F(i,k,j)=0 , \ \  \ \ iii) \ F(k,i,j)=0.$$}

The existence of the solution simply means that the constraint on the configurations is well posed and the uniqueness is simply the solvability of the implicit function equation.
Now we can formulate the following ``rigidity'' theorem.

\begin{theorem}
Let $p < q\in\mathbb{N}$, $(p,q)=1$, and let $F\colon\Sigma_2\times\Sigma_2\times\Sigma_2\to\mathbb{R}$ be a ``deterministic" function $F(x,y,z)$ depending only on the first coordinates of $x,y,z
\in \Sig_2$.
  Then recalling  the definition \eqref{defF} of $X_F$, one has 
  \begin{equation}
\label{dimtriangle}
\dim_H(X_F)=\dim_M(X_F)=\frac{q-1}q.
\end{equation}
\end{theorem}
\begin{proof}
A moment's thought makes it clear that for a function fulfilling i) any configuration in $\Delta_N^i$ is uniquely determined by choosing the $n(i,N)+1$ boundary values $(n,m)=(0,0), (1,0), \cdots ,( n(i,N),0)$ freely. For a function fulfilling ii), any configuration in $\Delta_N^i$ is uniquely determined by the free choice of the upper boundary values $(n,m)=(0,0), (1,1), (2,2),\cdots ,(m(i,N),m(i,N))$ and the pivotal ``squares''.  For a function fulfilling iii), any configuration is uniquely determined by the free choice of the values in $\partial_r^i(N)$. In all cases,  there are exactly $2^{n(i,N)+1}$ configurations   in $\Delta_N^i$.

The common dimension can easily be computed, we explain it for a function $F$ fulfilling i). Given an admissible sequence    $x\in X_F$, observe that we can choose all the values $x_i$ arbitrarily as long as $i$ is not divisible by $q$. The values at positions $qj$, $j \geq 1$, are then completely determined. Hence, exactly $\frac{q-1}q$ of the positions can be arbitrary and the rest is determined. This immediately implies, by standard argument, that
\[
 \dim_M(X_F)=\frac{q-1}q.
\] 
 The Minkowski and the Hausdorff dimension coincide as can be derived from part (ii) of Theorem \ref{th-var}.
\end{proof}

\begin{example}\label{e1}
The function  $ 
F(x,y,z)=(2x_1-1)(2y_1-1)(2z_1-1)-1$ is ``deterministic", hence \eqref{dimtriangle} holds.
This case is   reminiscent of the well-studied Ledrappier shift, i.e.  each of the triples $(x_\ell,x_{p\ell},x_{q\ell})$ has an even number of $0$'s, or equivalently
$$ x_\ell+ x_{ p\ell} + x_{q\ell}  = 1  \  \ \mbox{(mod 2)}.$$
  \end{example}

\begin{example}
For  $
F(x,y,z)=(y-xz)^2
$ (which  fulfills ii) but not i) or iii)) and   $F(x,y,z)=(x-yz)^2$ (which  fulfills iii) but not i) or ii)),  \eqref{dimtriangle} is satisfied. \end{example}

\begin{example}
Consider $F(x,y,z)=(x-y)^2+(x-z)^2. $ This function is not ``deterministic". In this case, each triangle $\Delta_N^i$ has exactly 2 configurations: all 0's or all 1's.
Hence the total number of cylinders of length $N$ equals the number of non-empty triangles $\Delta_N^i$. Since $(i,S)=1$ this number equals the number of $i$'s that are neither divisible by $p$ nor by $q$. We have exactly $p+q-1$ residue classes modulo $pq$ that are divisible by $p$ or $q$. Therefore the numbers $i$ with $(i,S)=1$ can be divided into $pq-p-q+1$ arithmetic sequences of step length $pq$. Hence
\[
\dim_M(X_F)=1-\frac{p+q-1}{pq}.
\]
In this case, the Hausdorff dimension of $X_F$ coincides with its Minkowski dimension, using the argument (ii) of Theorem \ref{th-var}: each   prefix  (with our interpretation, each finite triangle)  has only one possible continuation.
\end{example}

%


%
 

\end{document}